\documentclass[12pt]{amsart}
\usepackage{amssymb}
\newtheorem{thm}{Theorem}[section]
\newtheorem{lem}[thm]{Lemma}
\newtheorem{cor}[thm]{Corollary}
\newtheorem{prop}[thm]{Proposition}
\newtheorem{ex}[thm]{Example}

\newtheorem*{prob*}{Open problem}

\theoremstyle{definition}

\newtheorem{defi}[thm]{Definition}

\theoremstyle{remark}

\newtheorem{rem}[thm]{Remark}
\newtheorem*{rem*}{Remark}

\DeclareMathOperator{\id}{id}
\DeclareMathOperator{\s}{span}

\newcommand{\kringel}{\mathbin{\raise1pt\hbox{$\scriptstyle\circ$}}} 
\newcommand{\pkt}{\mathbin{\raise0pt\hbox{$\scriptstyle\bullet$}}}

\newcommand{\C}{\mathbb{C}}

\newcommand{\ad}{\mathop{\rm ad}}

\newcommand{\End}{\mathop{\rm End}}
\newcommand{\Der}{\mathop{\rm Der}}

\newcommand{\La}{\mathfrak{a}}
\newcommand{\Lb}{\mathfrak{b}}

\newcommand{\Lg}{\mathfrak{g}}
\newcommand{\Lh}{\mathfrak{h}}

\newcommand{\Ll}{\mathfrak{l}}

\newcommand{\Lr}{\mathfrak{r}}
\newcommand{\Ls}{\mathfrak{s}}

\newcommand{\Lu}{\mathfrak{u}}

\newcommand{\al}{\alpha}

\newcommand{\la}{\lambda}
\newcommand{\om}{\omega}

\newcommand{\Om}{\Omega}

\newcommand{\ov}{\overline}

\newcommand{\ra}{\rightarrow}

\renewcommand{\phi}{\varphi}

\begin{document}


\title[Novikov algebras]{Classical r-matrices and Novikov algebras}

\author[D. Burde]{Dietrich Burde}
\address{Fakult\"at f\"ur Mathematik\\
Universit\"at Wien\\
  Nordbergstr. 15\\
  1090 Wien \\
  Austria} 
\email{dietrich.burde@univie.ac.at}

\subjclass{Primary 17B30, 17D25}

\begin{abstract}
We study the existence problem for Novikov algebra structures on finite-dimensional Lie algebras.
We show that a Lie algebra admitting a Novikov algebra is necessarily solvable. Conversely we present a
$2$-step solvable Lie algebra without any Novikov structure. We use extensions and classical
$r$-matrices to construct Novikov structures on certain classes of solvable Lie algebras.
\end{abstract}

\maketitle

\section{Introduction}
A Novikov algebra is a special case of an LSA - a left-symmetric algebra.
It was introduced in the study of Hamiltonian operators concerning integrability 
of certain nonlinear partial differential equations \cite{GD}. It also appears in connection with Poisson
brackets of hydrodynamic type, operator Yang-Baxter equations \cite{BN} and vertex algebras \cite{FHL}. 
In particular, Novikov algebras bijectively correspond to a special class of Lie conformal
algebras. This shows the importance of Novikov algebras in theoretical physics. On the other hand
these algebras arise also in differential geometry: since an LSA is a Lie-admissible algebra the commutator 
defines a Lie algebra. The existence question asks which Lie algebras can arise that way.
This is a difficult question with a long history, see \cite{BU1}. It is related to affinely flat
manifolds and affine crystallographic groups. \\
We pose here the existence question for Novikov structures. This question is easier than the existence question
for LSA-structures. One reason is, that there are fewer Lie algebras admitting a Novikov
structure than Lie algebras admitting an LSA-structure. \\
In the first section we give some examples of Lie algebras admitting Novikov structures,
such as low-dimensional filiform Lie algebras. We show that a Lie algebra
admitting a Novikov structure has to be solvable. Then we give an example of a $2$-step 
solvable Lie algebra not admitting any Novikov structure.  
In the second section we construct Novikov structures via classical $r$-matrices.
The study of classical r-matrices was initiated in \cite{STS}. They arise also in the study
of differential Lie algebras and Poisson brackets.
In the third section we lift Novikov structures on an abelian Lie algebra $\La$ and on a
Lie algebra $\Lb$ to certain extensions of $\Lb$ by $\La$. This will be applied to prove the existence
of affine and Novikov structures on several classes of solvable and nilpotent Lie algebras.
For the proofs of some of the results we refer to our paper \cite{DBU}.

\section{Affine and Novikov structures}

An algebra $(A,\cdot)$ over $k$ with product $(x,y) \mapsto x\cdot y$
is called {\it left-symmetric algebra (LSA)}, if the product is
left-symmetric, i.e., if the identity
\begin{equation*}
x\cdot (y\cdot z)-(x\cdot y)\cdot z= y\cdot (x\cdot z)-(y\cdot x)\cdot z
\end{equation*}
is satisfied for all $x,y,z \in A$. The algebra is called {\it Novikov}, if
in addition
\begin{equation*}
(x\cdot y)\cdot z=(x\cdot z)\cdot y
\end{equation*}
is satisfied. We will always assume here that $k$ is a field of characteristic zero. 
Denote by $L(x), R(x)$ the right respectively left multiplication operator in
the algebra  $(A,\cdot)$. Then an LSA is a Novikov algebra if the right
multiplications commute:
\begin{align*}
[R(x),R(y)] & = 0.
\end{align*}

It is well known that LSAs are Lie-admissible algebras: the commutator
defines a Lie bracket. The associated Lie algebra then is said to admit
a left-symmetric structure, or affine structure.

\begin{defi}\label{affine}
An {\it affine structure} on a Lie algebra
$\Lg$ over $k$ is a left-symmetric product $\Lg \times \Lg \rightarrow \Lg$
satisfying 
\begin{equation}\label{lsa2}
[x,y]=x\cdot y -y\cdot x
\end{equation}
for all $x,y,z \in \Lg$. 
If the product is Novikov, we say that $\Lg$ admits a {\it Novikov structure}.
\end{defi}     
 
\begin{rem}
An affine structure on a Lie algebra $\Lg$ corresponds to a left-invariant
affine structure on a connected, simply connected Lie group $G$ with Lie algebra
$\Lg$. Such structures play an important role for affine crystallographic groups
and affine manifolds. In general, a Lie algebra need not admit an affine structure.
One may even find nilpotent Lie algebras which do not admit an affine structure, see \cite{BEN}, \cite{BU1}. 
On the other hand there are many classes of solvable and nilpotent Lie algebras which do
admit an affine structure: every positively graded Lie algebra and every $2$ and $3$-step 
nilpotent Lie algebra admits an affine structure, see Proposition $\ref{scheuneman}$.
\end{rem}

Let $\Lg$ be a Lie algebra. Denote the terms of the commutator series by 
$\Lg^{(1)}=\Lg$, $\Lg^{(i+1)}=[\Lg^{(i)},\Lg^{(i)}]$, and the terms of the lower central series 
by $\Lg^1=\Lg$, $\Lg^{i+1}=[\Lg,\Lg^{i}]$. 
A Lie algebra $\Lg$ is called  $p$-step solvable if $\Lg^{(p+1)}=0$. It is called $p$-step nilpotent, if $\Lg^{p+1}=0$.
A filiform nilpotent Lie algebra is a $p$-step nilpotent Lie algebra of dimension $n$ with $p=n-1$. \\[0.5cm]
In a Novikov algebra hold Jacobi-like identities:

\begin{lem}
Let $(A,\cdot)$ be a Novikov algebra. Then the following two identities hold for
all $x,y,z \in A$:
\begin{align*}
[x,y]\cdot z + [y,z]\cdot x + [z,x]\cdot y & = 0, \\
x \cdot [y,z] + y\cdot [z,x] + z\cdot [x,y] & = 0. \\
\end{align*}
\end{lem}
\begin{proof}
The first identity holds because we have
\begin{align*}
[x,y]\cdot z + [y,z]\cdot x + [z,x]\cdot y & =  (x\cdot y-y\cdot x)\cdot z+(y\cdot z-z\cdot y)\cdot x
+(z\cdot x-x\cdot z)\cdot y\\
 & = (x\cdot y)\cdot z + (y\cdot z)\cdot x + (z\cdot x)\cdot y \\
 & - (x\cdot z)\cdot y + (y\cdot x)\cdot z + (z\cdot y)\cdot x \\
 & = 0.
\end{align*}
The second identity follows similarly.
\end{proof}

The following result gives a necessary condition for the existence of Novikov structures.

\begin{prop}
Any finite-dimensional Lie algebra admitting a Novikov structure is solvable.
\end{prop}

\begin{proof}
Assume that $\Lg$ admits a Novikov structure given by
$(x,y)\mapsto x\cdot y$. Without loss of generality $k$ is algebraically closed.
Denote by $R(x)$ the right multiplication in the Novikov algebra $A$, i.e., $R(x)(y)=y\cdot x$. 
The algebra $A$ is called right-nilpotent if $R_A=\{R(x) \mid x\in A \}$ satisfies
$R_A^n=0$ for some $n\ge 1$. Let $I,J$ be two right-nilpotent ideals
in $A$. Since $[R(x),R(y)]=0$ the sum $I+J$ is also a right-nilpotent
ideal. Since $A$ is finite-dimensional there exists a largest right-nilpotent
ideal of $A$, denoted by $N(A)$. Now $N(A)$ is a complete left-symmetric
algebra since its right multiplications are nilpotent.
It is known that the Lie algebra of a complete LSA is solvable, see \cite{SEG}.
Hence the Lie algebra $\Lh$ of $N(A)$ is solvable.
On the other hand $A/N(A)$ is a direct sum of fields.
This was proved in \cite{ZEL}. It follows that the Lie algebra of $A/N(A)$ is abelian.
Hence $\Lg/\Lh$ is abelian, and $\Lh$ is solvable. It follows that $\Lg$ is solvable.
\end{proof}

Conversely, a solvable Lie algebra need not admit a Novikov structure in general, as we will
see in proposition $\ref{2.7}$.

\begin{prop}
Let $\Lg$ be a finite-dimensional two-step nilpotent Lie algebra.
Then $\Lg$ admits a Novikov structure.
\end{prop}

\begin{proof}
It is well known that for $x,y\in \Lg$ the formula $x\cdot y=\frac{1}{2}[x,y]$
defines a left-symmetric structure on $\Lg$. It satisfies $x\cdot(y\cdot z)=0$
for all $x,y,z\in \Lg$. Hence the product is also Novikov.
\end{proof}

Here is an example of a $2$-step solvable Lie algebra without a Novikov structure.

\begin{prop}\label{2.7}
The following $2$-step solvable Lie algebra does not admit any Novikov structure: 
let $\Lg$ be the free $4$-step nilpotent Lie algebra on $2$ generators $x_1$
and $x_2$. Let $(x_1,\ldots,x_8)$ be a basis and the Lie brackets as follows:
\begin{align*} 
x_3 & = [x_1,x_2] \\
x_4 & = [x_1,[x_1,x_2]] = [x_1, x_3]\\
x_5 & = [x_2,[x_1,x_2]] = [x_2,x_3] \\
x_6 & = [x_1,[x_1,[x_1,x_2]]] =  [x_1, x_4] \\
x_7 & = [x_2,[x_1,[x_1,x_2]]] = [x_2,x_4] \\
    & = [x_1,[x_2,[x_1,x_2]]] =  [x_1, x_5] \\
x_8 & = [x_2,[x_2,[x_1,x_2]]] =  [x_2,x_5]
\end{align*}
\end{prop}

\begin{proof}
Suppose that $\Lg$ admits a Novikov product with left multiplications $L(x)$. Then 
the right multiplications satisfy $R(x)= L(x)-\ad (x)$.
Suppose the operators are given by $L(x_k)=(x_{ij}^k)_{i,j}$ for $k=1,\ldots , 8$, with unknowns
$x_{ij}^k$. The operators $L(x)$ must satisfy the following conditions: 
\begin{align*}
[L(x),L(y)] & = L([x,y]), \\
 L([x,y])+\ad([x,y])& -[\ad(x),L(y)]-[L(x),\ad(y)] \\
 & = [L(x)-\ad x, L(y)-\ad y] \\
 & = [R(x),R(y)] \\
 & = 0
\end{align*}
for all $x,y\in \Lg$. These conditions are equivalent to a system of polynomial equations in the 
variables $x_{ij}^k$. In general, it is quite hard to solve such a system. But in our case, the condition
$L([x,y])+\ad([x,y]-[\ad(x),L(y)]-[L(x),\ad(y)]=0$ yields {\it linear} equations in the entries of
the operators $L(x_k)$. They reduce the system of all equations to very few equations -- which are
contradictory. This can be verified by a simple computer calculation. 
Note that $\Lg$ admits an affine structure since it is positively graded.
\end{proof}

\begin{prop}\label{2.11}
Let $\Lg$ be a finite-dimensional Lie algebra which is the direct vector space sum
$\Lg=\La\oplus \Lb$. Let $(e_1,\ldots ,e_n)$ be a basis of $\La$ and 
$(f_1,\ldots ,f_m)$ be a basis of $\Lb$. Suppose that
$[\La,\La]\subseteq \La$, $[\Lg,\Lb]\subseteq \Lb$, $[\Lg,[\La,\La]] =0$ and
$[\Lg,[\Lb,\Lb]] =0$. Then we obtain a Novikov structure on $\Lg$ by
\begin{align*}
e_i\cdot e_j & = \frac{1}{2}[e_i,e_j],\\
e_i\cdot f_j & = [e_i,f_j],\\
f_i\cdot e_j & = 0,\\
f_i\cdot f_j & = \frac{1}{2}[f_i,f_j].
\end{align*}
\end{prop}

The proof consists of a direct verification.

\begin{cor}\label{2.13}
Let $\Lg$ be a finite-dimensional Lie algebra admitting a $1$-codimensional ideal
$\Lb$ such that $[\Lb,\Lb]\subseteq Z(\Lg)$. Then $\Lg$ admits a Novikov structure.
\end{cor}

\begin{proof}
Let $\La=\langle e_1 \rangle$ and  $\Lb=\langle f_1,\ldots ,f_m \rangle$. Then Proposition $\ref{2.11}$ can
be applied. 
\end{proof}

To give more examples of Novikov structures on Lie algebras we consider filiform Lie algebras of dimensions 
$n\le 7$. We use the classification list of \cite{MAG}. Denote by  $\Lg_{I}(\al)$ with $\al\ne 0$ the
family of filiform Lie algebras defined by
\[
[x_1,x_i] = x_{i+1}; \; 2\le i\le 6;\quad
[x_2,x_3] =x_5; \; [x_2,x_4]=x_6; \; [x_2,x_5]=(1-\al)x_7;\; [x_3,x_4]=\al x_7.
\]

\begin{prop}
All complex filiform Lie algebras of dimension $n\le 7$ different from $\Lg_{I}(\al)$ admit a Novikov structure. 
An algebra $\Lg_{I}(\al)$ admits a Novikov structure if and only if $\al=\frac{1}{10}$.
\end{prop}

\begin{proof}
The filiform Lie algebras of dimension $n\le 5$ have an abelian commutator algebra.
Hence the claim follows by Corollary $\ref{2.13}$. All $6$-dimensional filiform algebras
can be written as\\
\begin{align*}
[x_1,x_i] & =  x_{i+1}; \; 2\le i\le 5;\\
[x_2,x_3] & = \al_1 x_5+\al_2 x_6; \\
[x_2,x_4] & = \al_1 x_6; \\
[x_2,x_5] & = -\al_3 x_6; \\
[x_3,x_4] & = \al_3x_6;\\
\end{align*}
with parameters $\al_1,\al_2,\al_3\in \C$. It is possible to construct a Novikov structure on these algebras
via $x_i.x_j =\la_{ij}[x_i,x_j]$ for $i\neq j$ with certain scalars $\la_{ij}$ relative to
our basis $(x_1,\ldots ,x_6)$. In general, of course, such a Novikov product need not exist, but
for the above filiform algebras we can find such a structure: \\
\begin{align*}
x_1\cdot x_i & = x_{i+1};\; 2\le i\le 5; & x_2\cdot x_5 & = -\frac{\al_3}{2}x_6; & x_4\cdot x_2 & =  -\frac{\al_1}{3}x_6; \\
x_2\cdot x_2 & = -\frac{\al_1}{3}x_4; & x_3\cdot x_2 & = -\frac{2\al_1}{3}x_5- \frac{\al_2}{2}x_6; & 
x_4\cdot x_3 & =  -\frac{\al_3}{2}x_6;\\
x_2\cdot x_3 & = \frac{\al_1}{3}x_5+ \frac{\al_2}{2}x_6; & x_3\cdot x_3 & =  -\frac{\al_1}{3}x_6; & 
x_5\cdot x_2 & = \frac{\al_3}{2}x_6;\\
x_2\cdot x_4 & = \frac{2\al_1}{3}x_6; & x_3\cdot x_4 & = \frac{\al_3}{2}x_6.\\
\end{align*}
The other products are equal to zero. One may easily verify this by hand, or by computer.
Similarly one can write the $7$-dimensional filiform algebras as\\
\begin{align*}
[x_1,x_i] & =  x_{i+1}; \; 2\le i\le 6;\\
[x_2,x_3] & = \al_1 x_5+\al_2 x_6+\al_3x_7; \\
[x_2,x_4] & = \al_1 x_6+\al_2x_7; \\
[x_2,x_5] & = (\al_1-\al_4) x_7; \\
[x_3,x_4] & = \al_4x_7;\\
\end{align*}
with parameters $\al_1,\al_2,\al_3,\al_4$. By computer calculations it is easy to
see that these algebras admit a Novikov structure if and only if 
\[
\al_1\al_4(10\al_4-\al_1)=0. 
\]
This condition can be always satisfied except for the algebra $\Lg_{I}(\al)$, which is obtained by the 
choice $\al_1=1,\al_2=0,\al_3=0$ and $\al_4=\al$. 
We have a Novikov product on $\Lg_{I}(\al)$, $\al\neq 0$, if and only if  $\al=\frac{1}{10}$.
\end{proof}

\section{Classical r-matrices and Novikov structures}

Let $\Lg$ be a Lie algebra, $\Lu$ be a $\Lg$-module and $T\colon \Lu\ra \Lg$ be a
linear map. Then we make $\Lu$ into an algebra by defining a skew-symmetric product $[,]_T$ by
\begin{align*}
[u,v]_T & = T(u).v-T(v).u. 
\end{align*}

\begin{defi}
The linear operator $T\colon \Lu\ra \Lg$ is called a {\it classical $r$-matrix}, or $T$-operator, if
$T$ is a homomorphism of algebras:
\begin{align}
T([u,v]_T) & = [T(u),T(v)].
\end{align}
This is equivalent to the fact, that $T$ satisfies the classical Yang-Baxter equation (CYBE)
\begin{align}\label{CYBE}
T(T(u).v-T(v).u)=[T(u),T(v)].
\end{align}
\end{defi}

\begin{rem}
Let $\Lu=\Lg$ be the adjoint module. Then we obtain the original definition
of a classical $r$-matrix and the CYBE of \cite{STS}. The CYBE is given in that case by
\begin{align*}
T([T(x),y]+[x,T(y)]) & = [T(x),T(y)].
\end{align*}
\end{rem}

If $T$ is a solution of CYBE, we can use $T$ to construct affine structures on Lie algebras.

\begin{thm}\label{3.3}
Let $\Lg$ be a Lie algebra, $\Lu$ be a $\Lg$-module and $T\colon \Lu\ra \Lg$ be a
linear map satisfying the CYBE. Then the product
\begin{align}\label{r7}
u\kringel v & = T(u).v
\end{align}
is left-symmetric. Hence $\Lu_T=(\Lu,[\, , \, ]_T)$ is a Lie algebra.
Then $\Lu_T$ admits an affine structure, and $T\colon \Lu_T\ra \Lg$
is a Lie algebra homomorphism.
\end{thm}

\begin{proof}
We will show that the above product is left-symmetric. Then the commutator 
\[
u\kringel v-v\kringel u =  T(u).v-T(v).u =[u,v]_T
\] 
automatically is a Lie bracket.
Let $u,v,w\in \Lu$. We have
\begin{align*}
(u,v,w) & = (u\kringel v)\kringel w-u\kringel (v\kringel w) = T(T(u).v).w-T(u).(T(v).w) \\
(v,u,w) & = (v \kringel u)\kringel w-v\kringel (u\kringel w) = T(T(v).u).w-T(v).(T(u).w). 
\end{align*}
Using the fact that $\Lu$ is a $\Lg$-module and \eqref{CYBE} we obtain
\begin{align*}
(u,v,w)-(v,u,w) & = T(T(u).v-T(v).u).w - (T(u).T(v)-T(v).T(u)).w\\
 & = (T(T(u).v-T(v).u)-[T(u),T(v)]).w \\
 & = 0.
\end{align*}
\end{proof}

\begin{cor}
Suppose that $T\colon \Lu\ra \Lg$ satisfies the CYBE and the condition 
\begin{align}\label{novbed}
T(T(u).v).w & = T(T(u).w).v
\end{align}
for all $u,v,w\in \Lu$. Then the product \eqref{r7} defines a Novikov structure on
$\Lu_T$.
\end{cor}

\begin{ex}
Let $\Lg=\Ls\Ll_2(\C)$ with $[x,y]=h$, $[x,h]=-2x$, $[y,h]=2y$
and $\Lu=\Lg$ be the adjoint module. Then
\[
T=
\begin{pmatrix}
0 & 0 & 0 \\
0 & 0 & 0 \\
0 & 0 & 1
\end{pmatrix}
\]
defines a Novikov structure on the solvable
Lie algebra $\Lu_{T} \cong \Lr_{3,-1}(\C)$.
\end{ex}
In fact, the Lie brackets of  $\Lu_{T}$ are given by
\begin{align*}
[x,h]_{T} & = -2x \\
[y,h]_{T} & = 2y.
\end{align*}

\begin{ex}\label{3.6}
Let $\Lg=\Ls\Ll_2(\C)$ and $\Lu$ be the natural $\Lg$-module
with basis $(v_0,v_1)$, i.e., with

\begin{align*}
x.v_0 & = 0,\quad y.v_0=v_1, \quad h.v_0=v_0 \\
x.v_1 & = v_0 \quad y.v_1 = 0,\quad h.v_1=-v_1
\end{align*}
Then, for all $c_1,c_2 \in \C$,
\[
T=
\begin{pmatrix}
0 & 0  \\
c_2 & c_1  \\
c_1 & 0
\end{pmatrix}
\]
defines a Novikov structure on $\Lu_{T}$.
\end{ex}

We have $T(v_0)= c_2y + c_1h$ and $T(v_1) = c_1y$, so that
\begin{align*}
[v_0,v_1]_{T} & = T(v_0).v_1-T(v_1).v_0 = -2c_1v_1.
\end{align*}
Hence we have $\Lu_{T}\cong \C^2$ for $c_1=0$ and
$\Lu_{T}\cong \Lr_2(\C)$ otherwise.

\begin{prop}
Let $\phi\colon \Lu'\ra \Lu$ be a $\Lg$-module homomorphism. If
$T\colon \Lu\ra \Lg$ satisfies the CYBE then so does $T^{\prime}\colon \Lu'\ra \Lg$,
where $T'=T\phi$. Moreover $\phi$ induces a Lie algebra homomorphism $\ov{\phi}\colon \Lu'_{T^{\prime}}\ra \Lu_{T}$.
In particular, if $\phi$ is a module isomorphism, $\Lu'_{T^{\prime}}\cong \Lu_{T}$ as Lie algebras.
\end{prop}

\begin{proof}
We have $\phi(x.u)=x.\phi(u)$ for $x\in \Lg$ and $u\in \Lu'$.
Since $T$ satisfies CYBE it follows
\begin{align*}
T^{\prime} (T^{\prime} (u).v-T^{\prime}(v).u) & = T\phi (T(\phi(u)).v-T(\phi(v)).u) \\
 & = T(T(\phi(u)).\phi(v)-T(\phi(v)).\phi(u)) \\
 & = [T(\phi(u)),T(\phi(v))] \\
 & = [T^{\prime} (u), T^{\prime} (v)].
\end{align*}
In other words, $T^{\prime}$ also satisfies CYBE. Hence $(\Lu,[\, , \, ]_T)$ and $(\Lu',[\, , \, ]_{T^{\prime}})$
are Lie algebras, and
\begin{align*}
\phi([u,v]_{T^{\prime}}) & = \phi(T\phi(u).v-T\phi(v).u) \\
 & = T\phi(u).\phi(v)-T\phi(v).\phi(u) \\
 & = [\phi(u),\phi(v)]_T.
\end{align*}
\end{proof}

Let $\Lh$ be a Lie algebra and $(\Lg,\Lu,T)$ be a triple
consisting of a Lie algebra $\Lg$, a $\Lg$-module $\Lu$ and a linear map
$T\colon \Lu\ra \Lg$ satisfying the CYBE. We call  $(\Lg,\Lu,T)$ a {\it Yang-Baxter-triple} for $\Lh$,
if $\Lh \cong \Lu_T$ by Theorem $\ref{3.3}$. The question is, for which Lie algebras
$\Lh$ exists a Yang-Baxter-triple. Let $A$ be a given LSA with underlying Lie algebra $\Lg$.
The left-multiplication in $A$ defines a $\Lg$-module $\Lg_L$. Let $\id\colon \Lg_L\ra \Lg$ be the
the identity map. Then $T=\id$ satisfies the CYBE: 
\begin{align*}
\id(\id (x).y-\id(y).x ) & = x.y-y.x =[x,y]=[\id(x),\id(y)]. 
\end{align*}
Hence $(\Lg,\Lg_L,\id)$ is a Yang-Baxter-triple for $\Lg$, i.e., any affine structure on a Lie algebra
arises from a Yang-Baxter triple.
The following proposition shows how to construct Yang-Baxter-triples in certain cases.

\begin{prop}
Let $\Lg$ be a Lie algebra with basis $(x_1,\ldots ,x_n)$, $\Lu$ be a $\Lg$-module with basis
$(u_1,\ldots ,u_n)$. Fix $\ell,k\in \{ 1,\ldots ,m \}$ to define a linear map
$T\colon \Lu\ra \Lg$ by
\begin{align*}
T(u_i)& =
\begin{cases}
x_k, & \text{if $i=\ell$},\\
0  , & \text{if $i\neq \ell$}.
\end{cases}
\end{align*}
Assume that $T(x_k.u_j) = 0 \text{ for all } j=1,\ldots ,m$. Then $T$ satisfies \eqref{CYBE} and \eqref{novbed}, 
and hence defines a Novikov structure on $\Lu_T$ by $u_i\kringel u_j = T(u_i).u_j$. The Lie bracket on 
$\Lu_T$ is given by
\begin{align*}
[u_i, u_j]_T & = 
\begin{cases}
x_k.u_j, & \text{if $i=\ell, j\neq \ell$},\\
-x_k.u_i, & \text{if $i\neq \ell, j=\ell$},\\
0  & \text{otherwise}.
\end{cases}
\end{align*}
\end{prop}

\begin{proof}
We have to verify the following two conditions: 
\begin{align*}
T(T(u_i).u_j-T(u_j).u_i) & = [T(u_i),T(u_j)], \\
T(T(u_i).u_j).u_r & = T(T(u_i).u_r).u_j.
\end{align*}
For the first identity, suppose that $j\neq \ell$. Then $T(u_j)=0$. Since $T(u_i)=x_k$ for $i=\ell$,
$T(u_i)=0$ for $i\neq \ell$ and  $T(x_k.u_j)=0$ by assumption, we have 
$T(T(u_i).u_j)=0$, and the claim follows. Now suppose that
$j=\ell$. We have to show that $T(T(u_i).u_{\ell})=[T(u_i),x_k]$. For $i=\ell$ both sides are equal to zero.
If $i\neq \ell$, then $T(u_i)=0$ and the equation reduces to $T(-x_k.u_i)=0$, which is true by assumption.\\
To show the second identity, assume that $i=\ell$. Then both sides are equal to zero. The other case,
$i\neq \ell$ is also clear.
\end{proof}

\begin{ex}
Let $\Lg=\Ls\Ll_2(\C)$ and $\Lu=V(2)$ be the natural module, see example $\ref{3.6}$. 
We change the notation according to the above proposition:
let $(x_1,x_2,x_3)=(x,y,h)$ be the basis of $\Lg$ and $(u_1,u_2)$ be the basis of $\Lu$.
Fix $\ell=2, k=1$, i.e., $T(u_1)=0$ and $T(u_2)=x_1$.
Then $T(x_k.u_j)=T(x_1.u_j)=0$ for $j=1,2$, and 
\[
T=
\begin{pmatrix}
0 & 1  \\
0 & 0  \\
0 & 0
\end{pmatrix}
\]
defines a Novikov structure on $\Lu_{T}\cong \C^2$.
\end{ex}

We want to mention also the following well known construction, see \cite{DOG}.

\begin{prop}\label{3.10}
Let $(A,.)$ be an associative, commutative algebra and $D$ a derivation of $A$, i.e.,
satisfying $D(x.y)=D(x).y + x.D(y)$. Then the product $x\kringel y=x.D(y)$ is Novikov.
In particular, it defines a Novikov structure on the Lie algebra $\Lg$ given by
\[
[x,y]:=x\kringel y-y\kringel x = x.D(y)-y.D(x).
\]
\end{prop}

Note that $D\in\Der (A)$ implies $D\in \Der (\Lg)$:
\begin{align*}
D([x,y]) & = D(x.D(y)-y.D(x)) \\
 & = D(x).D(y)+x.D^2(y)-D(y).D(x)-y.D^2(y) \\
 & = [D(x),y]+[x,D(y)].
\end{align*}

\begin{ex}
The free $3$-step nilpotent Lie algebra $\Lg$ of dimension $5$ with brackets
$[x_1,x_2]=x_3$, $[x_1,x_3]=x_4$ and $[x_2,x_3]=x_5$ admits a Novikov structure by
proposition $\ref{3.10}$.
\end{ex}

Define the commutative and associative product for $A$ by
\begin{align*}
x_1.x_1 & = x_1+x_2, & x_2.x_1 & = x_2, & x_3.x_1 & = x_3,& x_4.x_1 & = x_4+ x_5/2, & x_5.x_1 & = x_5,\\
x_1.x_2 & = x_2, & x_2.x_2 & = 0,& x_3.x_2 & = 0, & x_4.x_2 & = x_5/2, & x_5.x_2 & = 0, \\
x_1.x_3 & = x_3, & x_2.x_3 & = 0,& x_3.x_3 & = -x_5/2, & x_4.x_3 & = 0,& x_5.x_3 & = 0, \\
x_1.x_4 & = x_4+x_5/2, & x_2.x_4 & = x_5/2, & x_3.x_4 & = 0,& x_4.x_4 & = 0,& x_5.x_4 & = 0, \\
x_1.x_5 & = x_5, & x_2.x_5 & = 0, & x_3.x_5 & = 0,& x_4.x_5 & = 0,& x_5.x_5 & = 0.
\end{align*}

Then choose $D\in \Der (A)$ as
\[
D=
\begin{pmatrix}
0 & 0 & 0 & 0 & 0 \\
0 & 0 & 0 & 0 & 0 \\
1 & 1 & 0 & 0 & 0 \\
0 & 0 & 1 & 0 & 0 \\
0 & 0 & -1 & 0 & 0 
\end{pmatrix}.
\]
Then $x_i.D(x_j)=x_i\kringel x_j$ is a Novikov structure on $\Lg$ with
$x_1\kringel x_1=x_3$, $x_1\kringel x_2=x_3$, $x_1\kringel x_3=x_4-x_5/2$, $x_2\kringel x_3=x_5/2$,
$x_3\kringel x_1=-x_5/2$, $x_3\kringel x_2=-x_5/2$.

\section{Construction of Novikov structures via extensions}

In the following we will consider Lie algebras $\Lg$ which are an extension
of a Lie algebra $\Lb$ by an abelian Lie algebra $\La$. Hence we have a
short exact sequence of Lie algebras
\begin{equation*}
0 \rightarrow \La \xrightarrow{\iota} \Lg \xrightarrow{\pi}
\Lb \rightarrow 0
\end{equation*}
Since $\La$ is abelian, there exists a natural $\Lb$-module structure
on $\La$. We denote the action of $\Lb$ on $\La$ by $(x,a)\mapsto \phi (x)a$,
where $\phi \colon \Lb \ra \End (\La)$ is the corresponding Lie algebra
representation. We have
\begin{equation}
\phi([x,y])=\phi(x)\phi(y)-\phi(y)\phi(x)\label{5}
\end{equation}
for all $x,y \in \Lb$.
Let $\Om \in Z^2(\Lb,\La)$ be a $2$-cocycle. This means that
$\Om : \Lb \times \Lb \ra \La$ is a skew-symmetric bilinear map
satisfying
\begin{equation}\label{6}
\phi(x)\Om(y,z)-\phi(y)\Om(x,z)+\phi(z)\Om (x,y) =\Om ([x,y],z)-\Om ([x,z],y)+\Om([y,z],x).
\end{equation}
We obtain a Lie bracket on $\Lg=\La \times \Lb$ by
 \begin{equation}\label{lie-algebra}
[(a,x),(b,y)]:=(\phi (x)b-\phi(y)a+\Om(x,y),[x,y])
\end{equation}
for $a,b\in \La$ and $x,y\in \Lb$.
As a special case we obtain the $2$-step solvable Lie algebras as extensions
of two abelian Lie algebras $\La$ and $\Lb$. 
Now we want to construct affine structures on Lie algebras $\Lg$ which are
an extension of a Lie algebra $\Lb$ by an abelian Lie algebra $\La$.
Assume that $\Lg=(\La,\Lb,\phi,\Om)$ is an extension with the above data.
Suppose that we have already an LSA-product  $(a,b)\mapsto a\cdot b$ on
$\La$ and an LSA-product $(x,y)\mapsto x\cdot y$ on $\Lb$.
Since  $\La$ is abelian, the LSA-product on $\La$ is commutative and
associative. Hence it defines also a Novikov structure on $\La$.
We want to lift these  LSA-products to $\Lg$. Consider
\begin{align*}
\om & \colon \Lb \times \Lb \ra \La\\
\phi_1,\, \phi_2 & \colon \Lb \ra\End (\La)
\end{align*}
where $\om$ is a bilinear map and 
$\phi_1,\, \phi_2$ are Lie algebra representations.  
We will define a bilinear product $\Lg \times \Lg \ra \Lg$ by
\begin{equation}\label{produkt}
(a,x)\kringel (b,y):=(a\cdot b+\phi_1(y)a+\phi_2 (x)b+\om(x,y),x\cdot y)
\end{equation}  
It is straightforward to check the conditions for this product to be left-symmetric
or Novikov, see \cite{DBU}. The result is:

\begin{prop}\label{phi12}
The above product defines a left-symmetric structure on $\Lg$ 
if and only if the following conditions hold:
\begin{align}
\om(x,y)-\om(y,x) & = \Om(x,y) \label{8}\\
\phi_2(x)-\phi_1(x) & = \phi(x)\label{9}\\
\phi_2(x)\om(y,z)- \phi_2(y)\om(x,z) - \phi_1(z)\Om(x,y) & =
\om(y,x\cdot z)-\om (x,y\cdot z) +\om ([x,y],z)\label{10}\\
a\cdot \om(y,z)+\phi_1 (y\cdot z)a & = \phi_2(y)\phi_1(z)a-\phi_1(z)\phi(y)a \label{11} \\
a\cdot (\phi_1(z)b) & = b\cdot (\phi_1(z)a) \label{12}\\
\phi_2(y)(a\cdot c)- a\cdot (\phi_2(y)c) & = (\phi(y)a)\cdot c \label{13} \\
\Om(x,y)\cdot c & =0\label{14}
\end{align}
for all $a,b,c \in \La$ and  $x,y,z \in \Lb$.
\end{prop}

\begin{prop}\label{novikov}
The product \eqref{produkt} defines a Novikov structure on $\Lg$ if and
only if the conditions for a left-symmetric structure are satisfied, and
in addition the following conditions hold:
\begin{align}
\phi_1(z)\om(x,y)-\phi_1(y)\om(x,z) & = \om(x\cdot z,y)-\om(x\cdot y,z) \label{15}\\
\om(x,y)\cdot c+\phi_2(x\cdot y)c & = \phi_1(y)\phi_2(x)c \label{16}\\
[\phi_1(x),\phi_1(y)] & = 0 \label{17}\\
(\phi_2(x)b)\cdot c & = (\phi_2(x)c)\cdot b \label{18}\\
\phi_1(z)(a\cdot b) & = (\phi_1(z)a)\cdot b \label{19}\\
(x\cdot y)\cdot z & = (x\cdot z)\cdot y \label{20} 
\end{align}
\end{prop}

If $\Lb$ is also abelian, then the Lie algebra $\Lg$ is two-step solvable.
We say that the LSA-products on $\La$ and $\Lb$ are trivial if
$a\cdot b=x\cdot y=0$ for all $a,b\in \La$ and $x,y\in \Lb$.

\begin{cor}\label{trivial}
Suppose that the LSA-products on $\La$ and $\Lb$ are trivial.
Hence $\Lb$ is also abelian.
Then \eqref{produkt} defines a left-symmetric structure
on $\Lg$ if and only if the following conditions hold:

\begin{align*}
\om(x,y)-\om(y,x) & = \Om(x,y)\\
\phi_2(x)-\phi_1(x) & = \phi(x)\\
\phi_2(x)\om(y,z)- \phi_2(y)\om(x,z)& = \phi_1(z)\Om(x,y)\\
[\phi_1(x),\phi_2(y)] & = \phi_1(x)\phi_1(y)
\end{align*}

It defines a Novikov structure on $\Lg$ if in addition
\begin{align*}
\phi_1(z)\om(x,y) & = \phi_1(y)\om(x,z)\\
\phi_1(x)\phi_2(y) & = 0\\
[\phi_1(x),\phi_1(y)] & = 0
\end{align*}
\end{cor}

In particular, if $\phi_1=0$, the product defines a Novikov
structure on $\Lg$ if and only if it defines a left-symmetric structure
on $\Lg$.

\begin{cor}
Assume that $\Lg= \La \rtimes_{\phi} \Lb$ is a semidirect product of an abelian
Lie algebra $\La$ and a Lie algebra $\Lb$ by a representation 
$\phi\colon \Lb \ra \End(\La)=\Der(\La)$. This yields a split exact sequence
\begin{equation*}
0 \rightarrow \La \xrightarrow{\iota} \Lg \xrightarrow{\pi}
\Lb \rightarrow 0.
\end{equation*}
If $\Lb$ admits an LSA-product then $\Lg$ also admits an LSA-product.
If $\Lb$ admits a Novikov product $(x,y)\mapsto x\cdot y$ such that
$\phi(x\cdot y)=0$ for all $x,y\in \Lb$ then also $\Lg$ admits a Novikov product.
\end{cor}

\begin{proof}
Because the short exact sequence is split, the $2$-cocycle $\Om$ in the Lie bracket
of $\Lg$ is trivial, i.e., $\Om(x,y)=0$. Let $a\cdot b=0$ be the trivial product
on $\La$ and take $\phi_1=0$, $\om(x,y)=0$. Assume that $(x,y)\mapsto x\cdot y$ is
an LSA-product. Then all conditions of Proposition $\ref{phi12}$ are satisfied.
Hence \eqref{produkt} defines an LSA-product on $\Lg$, given by
\[
(a,x)\kringel (b,y)=(\phi(x)b, x\cdot y)
\]
Assume that the product on $\Lb$ is Novikov. Then the product on $\Lg$ will be Novikov
if and only if the conditions of Proposition $\ref{novikov}$ are satisfied.
In this case, only condition \eqref{16} remains, which says $\phi(x\cdot y)=0$. 
\end{proof}

\begin{ex}\label{ex-3.5}
Let $\Lg$ be the $5$-dimensional Lie algebra with basis $(A,B,C,X,Y)$
and Lie brackets $[X,Y]=A,\; [X,A]=B,\; [Y,A]=C$. This is a free 
$3$-step nilpotent Lie algebra. Hence it is $2$-step solvable.
The product \eqref{produkt} defines a Novikov structure on $\Lg$ with 

\begin{align*}
\phi_1(X) & = \begin{pmatrix} 0 & 0 & 0\\ -1/2 & 0 & 0\\ 0 & 0 & 0 \end{pmatrix},\quad
\phi_1(Y)=\begin{pmatrix} 0 & 0 & 0\\ 0 & 0 & 0\\ 0 & 0 & 0 \end{pmatrix} \\
\phi_2(X) & = \begin{pmatrix} 0 & 0 & 0\\ 1/2 & 0 & 0\\ 0 & 0 & 0 \end{pmatrix},\quad
\phi_2(Y)= \begin{pmatrix} 0 & 0 & 0\\ 0 & 0 & 0\\ 1 & 0 & 0 \end{pmatrix} \\
\end{align*}
and $\om (Y,X)=-A,\; \om(X,X)= \om(X,Y)= \om(Y,Y)=0$, and trivial products on
$\La$ and $\Lb$.
\end{ex}

The Lie algebra $\Lg$ is given by
$\Lg=(\La,\Lb,\phi,\Om)$ with $\La=\s \{A,B,C\},\; \Lb=\s \{X,Y\}$ abelian,
$\Om(X,Y)=A$ and $\phi(X)A=B,\;\phi(Y)A=C $, i.e.,
$$
\phi(X) = \begin{pmatrix} 0 & 0 & 0\\ 1 & 0 & 0\\ 0 & 0 & 0 \end{pmatrix},\quad
\phi(Y) = \begin{pmatrix} 0 & 0 & 0\\ 0 & 0 & 0\\ 1 & 0 & 0 \end{pmatrix}
$$

It is easy to see that the conditions of corollary $\ref{trivial}$ are satisfied.
Note that any product $\phi_i(x)\phi_j(y)=0$ for $1\le i,j\le 2$ and
$x,y\in \Lb$. We can write down the resulting Novikov structure on
$\Lg$ explicitely. It is given by \eqref{produkt}; the non-zero
products are given by \\
\begin{align*}
A\kringel X & = -B/2, \quad X\kringel A = B/2, \quad Y\kringel A = C, \quad Y\kringel X = -A.\\
\end{align*}

\begin{prop}\label{iso}
Let $\Lg=(\La,\Lb,\phi,\Om)$ be a two-step solvable Lie algebra.
If there exists an $e\in \Lb$ such that $\phi(e)\in\End(\La)$ is an
isomorphism, then $\Lg$ admits a Novikov structure. In fact, in that
case \eqref{produkt} defines a Novikov product, where $\phi_1=0,
\, \phi_2=\phi$, the product on $\La$ and $\Lb$ is trivial, and
\begin{align*}
\om(x,y) & = \phi(e)^{-1}\phi(x)\Om(e,y)
\end{align*}
\end{prop}

\begin{proof}
We have to show that the conditions of corollary $\ref{trivial}$ are satisfied.
Applying $\phi(e)^{-1}$ to \eqref{6} with $z=e$ it follows
$\Om(x,y)-\phi(e)^{-1}\phi(x)\Om(e,y)+\phi(e)^{-1}\phi(y)\Om(e,x)=0$. This
just means that $\Om(x,y)=\om(x,y)-\om(y,x)$.
Furthermore we have
\begin{align*}
\phi(x)\om(y,z)-\phi(y)\om(x,z) & = \phi(x)\phi(e)^{-1}\phi(y)\Om(e,z)-
\phi(y)\phi(e)^{-1}\phi(x)\Om(e,z) \\
 & = \phi(e)^{-1}(\phi(x)\phi(y)-\phi(y)\phi(x))\Om(e,z)\\
 & = 0
\end{align*}
Hence the product defines a left-symmetric structure.
Since $\phi_1=0$ the structure is also Novikov.
\end{proof}

We note that we can obtain by proposition $\ref{phi12}$
a new proof of Scheuneman's result (see \cite{SCH}):

\begin{prop}\label{scheuneman}
Any three-step nilpotent Lie algebra admits an LSA structure.
\end{prop}

More generally we have shown \cite{DBU}:

\begin{prop}
Let $\Lg$ be a 2-step solvable Lie algebra with 
$\Lg^r=\Lg^4$ for all $r\ge 5$. Then $\Lg$ admits a complete left symmetric structure.
\end{prop}

It is natural to ask whether any $3$-step nilpotent Lie algebra admits a
Novikov structure. We have shown in \cite{DBU} that all $3$-step nilpotent Lie algebras with
$2$ or $3$ generators admit a Novikov structure. 
In general however, a Novikov structure may not exist on a $3$-step nilpotent Lie algebra.

\end{document}